\documentclass[11pt]{article}

\usepackage{graphicx}%
\usepackage{multirow}%
\usepackage{amsmath,amssymb,amsfonts}%
\usepackage{amsthm}%
\usepackage{mathrsfs}%

\usepackage{mathrsfs}%
\usepackage[title]{appendix}%
\usepackage{xcolor}%
\usepackage{textcomp}%
\usepackage{manyfoot}%
\usepackage{booktabs}%
\usepackage{algorithm}%
\usepackage{algorithmicx}%
\usepackage{algpseudocode}%
\usepackage{listings}%

\usepackage{hyperref}

\allowdisplaybreaks

\newtheorem{thm}{Theorem}%

\numberwithin{equation}{section} \overfullrule 5pt

\newtheorem{prop}[thm]{Proposition}

\newtheorem{lem}[thm]{Lemma}
\theoremstyle{remark}
\newtheorem{rmk}{Remark}

\date{ }

\begin{document}

\title{Ellipsephic harmonic series revisited}

\author{Jean-Paul Allouche \\
CNRS, IMJ-PRG, Sorbonne, 4 Place Jussieu \\
F-75252 Paris Cedex 05, France \\
{\tt jean-paul.allouche@imj-prg.fr} \\
\and
Yining Hu \\
Institute for Advanced Study in Mathematics \\
Harbin Institute of Technology \\
Harbin, PR China \\
{\tt huyining@protonmail.com} \\
\and
Claude Morin \\
{\tt claude.morin2@gmail.com}
}
\maketitle

\abstract {Ellipsephic or Kempner-like harmonic series are series of inverses of integers whose expansion 
in base $B$, for some $B \geq 2$, contains no occurrence of some fixed digit or some fixed block of digits. 
A prototypical example was proposed by Kempner in 1914, namely the sum inverses of integers whose 
expansion in base $10$ contains no occurrence of a nonzero given digit. Results about such series 
address their convergence as well as closed expressions for their sums (or approximations thereof). 
Another direction of research is the study of sums of inverses of integers that contain only a given finite 
number, say $k$, of some digit or some block of digits, and the limits of such sums when $k$ goes to infinity. 
Generalizing partial results in the literature, we give a complete result for any digit or block of digits in any base. 


}

\maketitle

\section{Introduction}

While the harmonic series $\sum \frac{1}{n}$ is divergent, restricting the indices in the sum to integers
satisfying innocent-looking conditions can yield convergent series. One of the first such examples is 
probably the 1914 result of Kempner \cite{Kempner} stating that the sum of inverses of integers whose 
expansion in base $10$ contains no occurrence of a given digit ($\neq 0$) converges.
After Kempner's paper and the 1916 paper of Irwin \cite{Irwin}, several papers addressed extensions 
or generalizations of this result, as well as closed forms of numerical computations of the sums of the
corresponding series: see, e.g., 
\cite{Alexander,Baillie1, Baillie2, Behforooz, Boas, Craven, Farhi, Fischer, Gordon, Klove, KS, LP, MS, Nathanson1, Nathanson2, Nathanson3, SB, SLF, Wadhwa75, Wadhwa79, WW} 
and the references therein.

\medskip

We will revisit harmonic series with missing digits or missing blocks of digits: these series are called 
Kempner-like harmonic series or ellipsephic harmonic series in the literature (for the origin of the term 
{\it ellipsephic} coined by C. Mauduit, one can look at  \cite[p.~12]{Col} and \cite[Footnote, p.~6]{Hu-N}; 
also see the discussion in \cite{All-Morin}). More precisely, let $B$ an integer $\geq 2$ and let $a_{w,B}(n)$ 
denote the number of occurrences of the word (string) $w$ in the base $B$ expansion of the integer $n$ 
(where $|w|$ is the length of the string $w$), and $s_B(n)$ the sum of the digits of the base $B$ expansion 
of  $n$. It was proven by Farhi \cite{Farhi} that, for any digit $j$ in $\{0, 1, \dots, 9\}$, one has
$$
\lim_{k \to \infty} \sum_{\substack{n \geq 1 \\ a_{j,10}(n) = k}} \frac{1}{n}
= 10 \log 10.
$$
As explained in \cite{All-Morin}, a post on {\tt mathfun} asked for the value of 
$$
\displaystyle\lim_{k \to \infty} \sum_{\substack{n \geq 1 \\ s_2(n) = k}} \frac{1}{n}\cdot
$$
It was proved in \cite{All-Morin} that
\begin{equation}\label{sB}
\lim_{k \to \infty} \sum_{\substack{n \geq 1 \\ s_B(n) = k}} \frac{1}{n}
= \frac{2 \log B}{B-1}\cdot
\end{equation}
Thus we have of course
$$
\lim_{k \to \infty} \sum_{\substack{n \geq 1 \\ a_{1,2}(n) = k}} \frac{1}{n}
=
\lim_{k \to \infty} \sum_{\substack{n \geq 1 \\ s_2(n) = k}} \frac{1}{n} = 2 \log 2.
$$
Here we will evaluate all such series, where $a_{1,B}(n)$ is replaced with
$a_{w,B}(n)$, where $w$ is any block of digits in base $B$, by proving that
\begin{equation}\label{awB}
\lim_{k \to \infty}\sum_{\substack{n \geq 1 \\ a_{w,B}(n) = k}} \frac{1}{n}  = B^{|w|} \log B.
\end{equation}
To prove this result, we will replace $1/n$ with the seemingly more complicated function 
$\log_B(\frac{n}{n+1})$, and make use of results proved in (or inspired by) \cite{Allouche-Shallit1989, Hu-Y}. 
In passing we will generalize \cite{Allouche-Shallit1989,All-Morin} and re-prove \cite{Hu-Y}. 

\section{``Reducing'' the problem}

Let $B$ be an integer $\geq 2$. Let $w$ be a string of letters in $\{0, 1, \dots, B-1\}$. Let $a_{w,B}(n)$
be the number of (possibly overlapping) occurrences of of the string $w$ in the base $B$ expansion of $n$.
First we note that the series $\displaystyle\sum_{\substack{n \geq 1 \\ a_{w,B}(n) = k}} \frac{1}{n}$ 
converges: the proof is the same as in \cite[Lemma~1, p.~194]{Allouche-Shallit1989}, namely one uses a
counting argument for the case of a single digit, and one replaces the base with some of its powers for 
the case of a block of digits. Now, to evaluate the series, the idea is to replace it with a convergent series
$\displaystyle\sum_{\substack{n \geq 1 \\ a_{w,B}(n) = k}} b_w(n)$ whose sum, say $A_w(k)$, 
tends to a limit, say $A_w$ when $k \to \infty$. Furthermore, if we have the property
$b_w(n) - 1/(B^{|w|}n) = {\mathcal O}_w(1/n^2)$ when $n$ tends to infinity, then we obtain
\begin{itemize}
\item[]{} $$\sum_{\substack{n \geq 1 \\ a_{w,B}(n) = k}} \frac{1}{n} \ \text{converges, and \ \ }$$
\item[]{} $$\lim_{k \to \infty} \sum_{\substack{n \geq 1 \\ a_{w,B}(n) = k}} \frac{1}{n} = B^{|w|} A_w.$$
\end{itemize}
(Note that if $a_{w,B}(n)$ tends to infinity, then $n$ must also tend to infinity.)

	Inspired by \cite{Allouche-Shallit1989, Hu-Y}, we define $L(n)$ by $L(0) := 0$ and
	$L(n) := \log_B\left(\dfrac{n}{n+1}\right)$ for $n\geq 1$.
	For a string $w$ over the alphabet $[0, B-1]$, let $v(w)$ denote
	the integer whose expansion in base $B$ is $w$ (with possible leading $0$'s if $w \neq 0^j$).

\begin{prop}\label{prop:minus1}
	Let $w$ be a nonempty string over the alphabet $[0, B-1]$,
	$$g=B^{|w|-1},\quad h=\left\lfloor \frac{v(w)}{B}\right\rfloor.$$
	Then, for all $k\geq 0$,
	\begin{equation}\label{eq:sum1}
	\sum_{\substack{n \\ a_{w,B}(gn+h)=k}} L(Bgn+v(w))=-1,
	\end{equation}
	where the sum is over $n\geq 1$ if $w=0^j$ and $n\geq 0$ otherwise.
\end{prop}
%<<<
\begin{proof}
	Let $c$ be the last letter of $w$. Let $d_w(k)$ be defined by 
	$$d_w(k)=\sum_{\substack{n\geq 0 \\ a_{w,B}(n)=k}} L(Bn+c).$$
	(Note that this series converges since $L(n) \sim \frac{1}{n \log B}$ when $n$ goes to infinity.)
	By writing $n=gr+m$, with $r\geq 0$ and $0\leq m\leq g-1$, we see that
	$$d_w(k)=\sum_{m=0}^{g-1}\sum_{\substack{r\geq 0 \\ a_{w,B}(gr+m)=k}} 
	L(Bgr+Bm+c).$$
	Similarly, if we let
	$$e_w(k)=\sum_{\substack{n\geq 0 \\ a_{w,B}(Bn+c)=k}} L(Bn+c)$$
	(which is convergent, like $d_w(k)$), then we have
	$$e_w(k)=\sum_{m=0}^{g-1}\sum_{\substack{r\geq 0 \\ 
	a_{w,B}(Bgr+Bm+c)=k}} L(Bgr+Bm+c).$$
	Note that 
	$$a_{w,B}(Bgr+Bm+c) - a_{w,B} (gr+m)=\begin{cases} 1 &\mbox{if } m=h\\
		0&\mbox{otherwise}
	\end{cases}
	$$
	for $r\geq 0$ if $w\neq 0^j$ and for $r\geq 1$ if $w=0^j$. For $w=0^j$
	and $r=0$, the above difference is $0$ for all $m$ because we do not pad 
	leading $0$'s in this case.
	
	\medskip
	
	\noindent
	Therefore
	\begin{multline}\label{eq:sum2}
	d_w(k)-e_w(k)=
	\sum_{\substack{r\\ a_{w,B}(gr+h)=k}} L(Bgr+v(w))
	\\
	- \sum_{\substack{r\\ a_{w,B}(gr+h)=k-1}} L(Bgr+v(w))
	\end{multline}
	where the sum is over $r\geq 1$ if $w=0^j$ and $r\geq 0$ otherwise.
	
	\noindent
	If we could show that $d_w(k)=e_w(k)$ for $k>0$, then it would follow from
	equation \eqref{eq:sum2} that the value of the sum
	$$\sum_{\substack{r\\ a_{w,B}(gr+h)=k}} L(Bgr+v(w))$$
	is independent of $k$ and hence equal to $d_w(0)-e_w(0)$. To prove this, notice that
	\begin{equation}\label{eq:n0}
	L(n)-\sum_{j=0}^{B-1} L(Bn+j)= \begin{cases}0 & \mbox{ if } n>0\\
		1 & \mbox{ if } n=0,
	\end{cases}
	\end{equation}
	
	\noindent
	and
	\begin{align*}
		\quad \sum_{\substack{n\geq 0\\ a_{w,B}(n)=k}} L(n)
		&=
	\sum_{j=0}^{B-1}\sum_{\substack{n\geq 0\\ a_{w,B}(Bn+j)=k}} L(Bn+j)\\
		&= \!\!\!
		\sum_{\substack{n\geq 0\\ a_{w,B}(Bn+c)=k}} L(Bn+c)
	+\sum_{\substack{j=0\\ j\neq c}}^{B-1}\sum_{\substack{n\geq 0\\ a_{w,B}(n)=k}} L(Bn+j).
	\end{align*}
	Hence 
	\begin{align*}
		e_w(k)&= \sum_{\substack{n\geq 0\\ a_{w,B}(Bn+c)=k}} L(Bn+c)\\
		&= \sum_{\substack{n\geq 0\\ a_{w,B}(n)=k}}  (L(n)-
	\sum_{\substack{j=0\\ j\neq c}}^{B-1}L(Bn+j))\\
		&= \sum_{\substack{n\geq 0\\ a_{w,B}(n)=k}}  (L(n)-
		\sum_{\substack{j=0 }}^{B-1}L(Bn+j))+
		\sum_{\substack{n\geq 0\\ a_{w,B}(n)=k}}  L(Bn+c)
	\end{align*}
	By \eqref{eq:n0}, the first sum is $1$ if $k=0$, and $0$  if $k>0$. 
	The second sum is the definition of $d_w(k)$.
\end{proof}
%>>>

\begin{lem}\label{lem:reduce}
	%<<<
	Let $t$ be an integer whose expansion in base $B$ is 
	$t=b_1b_2\ldots b_s$. For $1\leq r \leq s$, 

	If $b_1\ldots b_r$ is not a suffix of $w$, then
	$$
	\sum_{\substack{ n \\ a_{w,B}(B^rn+v(b_1\ldots b_r))=k}} L(B^sn+t)=
	\sum_{\substack{ n \\ a_{w,B}(B^{r-1}n+v(b_1\ldots b_{r-1}))=k}} L(B^sn+t).
	$$

	If $b_1\ldots b_r$ is  a suffix of $w$, then
	\begin{multline}
	\sum_{\substack{ n \\ a_{w,B}(B^rn+v(b_1\ldots b_r))=k}} L(B^s n+t)= \\
\sum_{\substack{ n \\ a_{w,B}(B^{r-1}n+v(b_2\ldots b_{r}))=k}} L(B^{s-1} n+t')
		\\ -
		\sum_{\substack{j=0\\ j\neq b_1 }}^{B-1}
		\sum_{\substack{ n \\ a_{w,B}(B^{r-1}n+v_j)=k}} L(B^sn+t_j)
	\end{multline}
	where $t'=v(b_2\cdots b_s)$, $t_j=v(jb_2\cdots b_s)$, and
	$v_j=v(jb_2\ldots b_{r-1})$ if $r\geq 2$ and $v_j=0$ if $r=1$.
\end{lem}

\medskip

The proof is the same as in \cite{Allouche-Shallit1989}.

\medskip

\begin{thm}\label{thm:general}
	There is a rational function $b_w(n)$ such that for all $k\geq 0$ we have
	\begin{equation}
		\sum_{\substack{n\\ a_{w,B}(n)=k}} \log_B(b_w(n)) = -1.
	\end{equation}
	(The summation is over $n\geq 1$ for $w=0^j$ and $n \geq 0$ otherwise.)
	and 
	\begin{equation}\label{eq:bworder}
		\log(b_w(n)) = -\frac{1}{B^{|w|}n} + \mathcal{O}(1/n^2).
	\end{equation}
\end{thm}

\begin{proof}
	The existence of $b_w(n)$ follows from Proposition \ref{prop:minus1} and
	iterated applications of Lemma \ref{lem:reduce}. 
	The process of obtaining $b_w(n)$ can be visualized by a tree $T$ whose
	root is 
	$$
	\sum_{\substack{n \\ a_{w,B}(gn+h)=k}} L(Bgn+v(w))
	$$
	and a node is a leaf if the condition of $n$ for sum is $a_{w,B}(n)=k$,
	has a child corresponding to the right side if we are in the first case in 
	Lemma \ref{lem:reduce}, and has $B$ children corresponding to the $B$ terms 
	(minus signs are included in the terms) of the right side if we are in the 
	second case. Then $b_w(n)$ is the sum of the summands of the leaves of this 
	tree.
	To prove \eqref{eq:bworder}, we first notice that
	$$\log B\cdot L(an+b)=\log\left(1-\frac{1}{an+b+1}\right)
	=-\frac{1}{an}+\mathcal{O}(1/n^2) $$
	where $a$ and $b$ are positive constants.
	In particular, in \eqref{eq:sum1}, 
	$$L(Bgn+v(w))=-\frac{1}{\log B\cdot B^{|w|}n}+\mathcal{O}(1/n^2).$$
	Then, we note that, in $T$, the first order term in the summand of each
	node that is not a leaf is the sum of the first order term in its children.
	For example, when a node has $B$ children and the first order term
	of the summand is
	$-\dfrac{1}{\log B\cdot B^sn}$, then the sum of first order terms of the
	summands of its children is 
	$$-\frac{1}{\log B\cdot B^{s-1}n} - \sum_{\substack{j=0\\ j\neq b_1}}^{B-1}
	\left(-\frac{1}{\log B \cdot B^sn}\right)= -\frac{1}{\log B\cdot B^sn}.$$
	By induction we conclude that
	\begin{equation*}
	\log_B(b_w(n))=  -\frac{1}{\log B\cdot B^{|w|}n}+\mathcal{O}(1/n^2).
	\qedhere
	\end{equation*}
\end{proof}

\begin{rmk}
Theorem \ref{thm:general} above generalizes the case $B = 2$ in \cite{Allouche-Shallit1989} (also see
\cite{Allouche-Hajnal-Shallit}). It can also give another proof of \cite[Theorem~3]{Hu-Y}.
\end{rmk}

\medskip

\begin{rmk}
Actually the same ``reducing trick'' can be used to re-prove Equation~(\ref{sB}) by using a result in
\cite{ACMFS}. Namely, up to notation, it was proved in \cite[Lemme, p.~142]{ACMFS} that, for all $k \geq 0$,
\begin{equation}\label{log}
\sum_{s_B(n) = k} \log\left(\frac{n+1}{B\lfloor n/B \rfloor + B}\right) = - \log B.
\end{equation}
Define the fractional part of $n/B$ by $\{n/B\} := n/B - \lfloor n /B \rfloor$. Then, we have when $n$ tends 
to infinity,
$$
\begin{array}{lll}
\displaystyle\log\left(\frac{n+1}{B\lfloor n/B \rfloor + B}\right) 
&=& \displaystyle\log\left(1 + \frac{1 - B + B\{n/B\}}{n+B(1 - \{n/B\})}\right) \\
&=& \displaystyle\log\left(1 + \frac{1 - B + B\{n/B\}}{n}\right) + O(1/n^2) \\
&=& \displaystyle\frac{1 - B}{n} + \left(\frac{B\{n/B\}}{n}\right) + O(1/n^2). \\
\end{array}
$$
Thus (convergences are consequences of, e.g., Equation~\ref{log}, see \cite{ACMFS}):  
\begin{align*}
- \log B &= \sum_{s_B(n) = k} \log\left(\frac{n+1}{B\lfloor n/B \rfloor + B}\right) \\
&= (1 - B) \sum_{s_B(n) = k} \frac{1}{n} \ + \sum_{s_B(n) = k} \frac{B\{n/B\}}{n} +  O(1/n^2).
\end{align*}
Hence (note that the term $O(1/n^2)$ below can be chosen independent of $k$),
\begin{align*}
- \log B & = (1 - B) \sum_{s_B(n) = k} \frac{1}{n} + \ 
\sum_{0 \leq j \leq B-1} \sum_{s_B(n) = k-j}\frac{j}{Bn+j} \ + O(1/n^2) \\
&= (1 - B) \sum_{s_B(n) = k} \frac{1}{n} + \
\sum_{0 \leq j \leq B-1} \sum_{s_B(n) = k-j}\frac{j}{Bn} \ + O(1/n^2). 
\end{align*}
Now, if $k$ tends to infinity, we have that $k-j$ tends to infinity for $j \in [0, B-1]$, and also that $n$ must 
tend to infinity, hence, letting $\lim_{k \to \infty}\sum_{s_B(n) = k} \frac{1}{n} := \ell$,
$$
- \log B = (1 - B) \ \ell+ \sum_{0 \leq j \leq B-1} \frac{j}{B} \ \ell = - \frac{B-1}{2} \ \ell, \ \ \
\text{thus \ \ $\ell = \frac{2 \log B}{B-1}\cdot$}
$$
\end{rmk}

\section{Statements and Declarations}

The authors have no competing interests. 
No funds, grants, or other support were received.

\end{document}